\documentclass[12pt]{article}
\usepackage[margin=1in]{geometry} 
\usepackage{amsmath,amsthm,amssymb,amsfonts}
\usepackage[normalem]{ulem}
\usepackage{enumitem,hyperref}
\usepackage{color}
\usepackage{comment}
\usepackage[noadjust]{cite}
\usepackage{graphicx}
\usepackage{authblk}

\newcommand{\cA}{\mathcal{A}}

\newcommand{\cL}{\mathcal{L}}
\newcommand{\cM}{\mathcal{M}}
\newcommand{\cP}{\mathcal{P}}

\newcommand{\bL}{\mathbb{L}}

\newcommand{\WF}{Q^{\text{WF}}}

\newcommand{\VP}{\text{VP}}
\newcommand{\WVP}{\text{WVP}}

\newcommand{\ZFC}{\text{ZFC}}

\newcommand{\id}{\text{id}}
\newcommand{\crit}{\text{crit}}

\newcommand{\ot}{\text{ot}}
\newcommand{\comp}{\text{comp}}
\newcommand{\I}{I}
\newcommand{\cof}{\text{cof}}
\newcommand{\Ord}{\text{Ord}}
\newcommand{\LST}{\text{LST}}
\newcommand{\ULST}{\text{ULST}}

\newcommand{\footnotei}[1]{}

\newtheorem{definition}{Definition}[section]
\newtheorem{lemma}[definition]{Lemma}
\newtheorem{theorem}[definition]{Theorem}
\newtheorem{corollary}[definition]{Corollary}

\newtheorem{claim}[definition]{Claim}
\newtheorem{proposition}[definition]{Proposition}
\newtheorem{question}[definition]{Question}
\newtheorem{fact}[definition]{Fact}

\begin{document}
	
	
	\title{Model-theoretic characterizations of large cardinals (Re)${}^2$visited\footnote{All results of this paper are also part of the second author's PhD thesis \cite{osinski2024}.\\This material is based upon work while the first author was supported by the National Science Foundation under Grant No. DMS-2137465 and DMS-2339018.}}
    \author[1]{Will Boney}
    \author[2]{Jonathan Osinski}
    \affil[1]{Mathematics Department, Texas State University, San Marcos, TX, USA}
    \affil[2]{Fachbereich Mathematik, Universität Hamburg, Hamburg, Germany}
	\date{\today}
 	\maketitle

    \begin{abstract}
        \noindent We characterize several large cardinal notions by model-theoretic properties of extensions of first-order logic. We show that $\Pi_n$-strong cardinals, and, as a corollary, ``Ord is Woodin" and weak Vopěnka's Principle, are characterized by compactness properties involving Henkin models for sort logic. This provides a model-theoretic analogy between Vopěnka's Principle and weak Vopěnka's Principle. We also characterize huge cardinals by compactness for type omission properties of the well-foundedness logic $\bL(\WF)$, and show that the compactness number of the Härtig quantifier logic $\bL(I)$ can consistently be larger than the first supercompact cardinal. Finally, we show that the upward Löwenheim-Skolem-Tarski number of second-order logic $\bL^2$ and of the sort logic $\bL^{s,n}$ are given by the first extendible and $C^{(n)}$-extendible cardinal, respectively.
    \end{abstract}

\tableofcontents
    
	\section{Introduction}
	Large cardinals are closely tied to model-theoretic properties of \emph{strong logics}, i.e., of extensions of first-order logic. Early on in the study of large cardinals, Tarski proposed weakly and strongly compact cardinals still prominent today as cardinals witnessing certain compactness principles for infinitary logics (cf. \cite{tar1962}). Subsequently, several other large cardinals such as supercompact and extendible cardinals received model-theoretic characterizations (cf. \cite{mag1971, ben1978}). Further, Vopěnka's Principle was shown to be equivalent to axiom schemas postulating that all logics haver certain downward Löwenheim-Skolem and compactness properties (cf. \cite{mag2011, mak1985}). In the last years, there has been a resurgence of this type of research and the large cardinal strengths of a series of model-theoretic properties of strong logics were considered (cf. \cite{mag2011, bagaria2016symbiosis, bon2020, gal2020, boney2024model, holy2024, gitman2024upward, osinski2024Henkin, osinski2024}).  

    While the above cited results on connections between Vopěnka's Principle and model-theoretic properties of logics provided some global equivalence, crucially, many of these newer results (e.g. from \cite{bon2020, gitman2024upward, osinski2024Henkin, osinski2024}) allow us to understand this relationship in a much more fine grained way. More concretely, the first author \cite{bon2020} showed that, for a natural number $n$, the existence of \emph{$C^{(n)}$-extendible cardinals} is equivalent to the existence of a \emph{compactness number} for the \emph{sort logic} $\bL^{s,n}$ (cf. Sections \ref{sec:AMT} and \ref{sec:LCs} for definitions). The $C^{(n)}$-extendible cardinals were introduced by Bagaria (cf. \cite{bag2012}) as a strengthening of the more familiar extendible cardinals and he showed that axioms stating their existence form, with rising $n$, a hierarchy of stronger and stronger axioms bounded by Vopěnka's Principle. Moreover, he showed that Vopěnka's Principle can be restricted in uniform ways such that these restrictions are precisely equivalent to the existence of $C^{(n)}$-extendible cardinals. The first author's results about compactness numbers of sort logic thus show that this hierarchy below Vopěnka's Principle is precisely mirrored by compactness properties of logics. Combining results of Poveda \cite[Theorem 5.2.3]{poveda2020} and the first author \cite[Theorem 4.13]{bon2020} one can show that a similar connection holds between Löwenheim-Skolem properties of sort logics and $C^{(n)}$-extendible cardinals. And strengthening the notions of \emph{Henkin models} for abstract logics (cf. Section \ref{sec:WVP}) to \emph{strong} Henkin models, Poveda and the second author \cite{osinski2024Henkin} showed  that certain compactness principles for strong Henkin models of sort logic provide yet another column to this hierarchy. 

	This paper can be seen as a follow up paper to \cite{bon2020} and \cite{boney2024model} and adds new model-theoretic characterizations of several large cardinals. This provides new insights into the relation of Vopěnka's Principle to model theory of strong logics. We will show that by considering \emph{upward} Löwenheim-Skolem properties of strong logics provides a fourth characterization of the hierarchy of large cardinal axioms below Vopěnka's Principle by model-theoretic means (cf. Section \ref{sec:ULST}). More concretely, we show that the \emph{ULST number} introduced in \cite{gal2020} of second-order logic $\bL^2$ is the first extendible cardinal. This confirms a conjecture from \cite{gal2020}. The result generalizes to sort logic, in the sense that the existence of a ULST number of the sort logic $\bL^{s,n}$ is equivalent to the existence of a $C^{(n)}$-extendible cardinal. As a corollary, we get a characterization of Vopěnka's Principle. 

    Moreover, we will consider what is known as \emph{weak Vopěnka's Principle}, which is an axiom schema motivated by a category-theoretic characterization of Vopěnka's Principle (cf. \cite{adamek1988}). Recently, Wilson \cite{wilson2020weak, wilson2022large} showed that weak Vopěnka's Principle is indeed strictly weaker than Vopěnka's Principle, solving a longstanding open problem. Bagaria and Wilson \cite{bagaria2023weak} showed that below weak Vopěnka's Principle, there is also hierarchy of stronger and stronger large cardinal axioms, stating the existence of what they call \emph{$\Pi_n$-strong} cardinals. There are thus analogous \emph{patterns} (cf. \cite{bagaria2024patterns}) in the large cardinal hierachy below Vopěnka's Principle and weak Vopěnka's Principle. Using compactness properties for Henkin models, in Section \ref{sec:WVP} we will give a characterization of $\Pi_n$-strong cardinals. This shows that the pattern below weak Vopěnka's Principle can also be obtained by model-theoretic properties of strong logics, as was the case for the pattern below Vopěnka's Principle, establishing a novel analogy between the two. As a corollary, we obtain a model-theoretic characterization of weak Vopěnka's Principle.

    We will further show that huge cardinals -- and thus, in particular, cardinals with higher consistency strength than Vopěnka's Principle -- can be characterized by a model-theoretic property of the logic $\bL(\WF)$, answering a question of Wilson (cf. Section \ref{sec:huge}). Finally, we will show how to separate the compactness number of the logic $\bL(I)$ from supercompact cardinals (cf. Section \ref{sec:compactnessLI}). 
	
	\section{Preliminaries}
	\subsection{Abstract model theory}\label{sec:AMT}
	We will use the following standard notions from abstract model theory. A \emph{language} is a set of (typically) finitary relation, function, and constant symbols. An \emph{abstract} or \emph{strong} logic $\cL$ consists of a definable class function that assigns to every language $\tau$ a set of \emph{sentences} $\cL[\tau]$, and of a definable relation $\models_\cL$ that may hold between $\tau$-structures $\mathcal A$ and members of $\cL[\tau]$. If $\mathcal A \models_\cL \varphi$, we say that $\mathcal A$ \emph{is a model of} $\varphi$. There are several standard niceness properties one assumes for the relation $\models_\cL$, e.g., that isomorphic structures are models of the same sentences. A precise definition of the properties needed can be found in \cite[Definition 2.1]{boney2024model} (which is itself a sublist of classical definitions). Note that we assume that $\cL[\tau]$ is always a set, as opposed to a proper class.
	
	If $\sigma$ and $\tau$ are languages, we say that a \emph{renaming} is a bijection $f: \sigma \rightarrow \tau$ which restricts to bijections between the sets of relation, function, and constant symbols in $\sigma$ and $\tau$, all while preserving their respective arities. We may turn a $\sigma$-structure $\mathcal A$ into a $\tau$-structure $f(\mathcal A)$ along $f$ in the obvious way. As part of the definition of an abstract logic, we demand that every renaming comes with a bijection\footnote{We use the same name for the renaming $\sigma \to \tau$ of languages and the induced maps of structures and formulas because there is no risk of confusion.}
 $f: \mathcal L[\sigma] \rightarrow \mathcal L[\tau]$ such that for any $\sigma$-structure and any $\varphi \in \mathcal L[\sigma]$, 
	\[
	\mathcal A \models_\cL \varphi \text{ if and only if } f(\mathcal A) \models_\cL f(\varphi).\footnote{
    The definition in \cite{boney2024model} includes the requirement that the function $f$ is unique. We do not need this uniqueness for our arguments, so we will not require this throughout the paper. Should one want to include uniqueness in the definition, one has to change the usual definitions of the concrete logics (like, say, second-order logic) one works with slightly: note that usually, the uniqueness requirement is not fulfilled, because given an $f$ as stated, one might always switch around two tautologies to obtain a second function $g$ with the desired property. It is however easy to obtain a version of the logic as needed in the following way: for every formula $\varphi$ of a logic $\cL$, consider the Scott equivalence class $[\varphi]$ consisting of all formulas of minimal rank equivalent to $\varphi$, and let $\cL[\tau]$ consist of these equivalence classes $[\varphi]$ instead of the formulas themselves.
    
    Furthermore, even if one sticks to the usual definitions and so the function $f$ is not unique, for the usual logics like second-order logic there is a canonical definition only dependent on the renaming between the vocabularies which identifies such an $f$.}
	\]
	If $f$ is a map as above and $T \subseteq \cL[\sigma]$, we call $f``T$ a \emph{copy} of the $\cL$-theory $T$.
	
	Let us mention the logics beyond first-order logic that we will work with. The logic $\bL(\WF)$ adds to first-order logic the quantifier $\WF$ with the semantics
	\[
	\mathcal A \models \WF xy \varphi(x,y) \text{ iff } \{(a,b) \in A^2 \colon \mathcal A \models \varphi(a,b) \} \text{ is well-founded.}
	\]
	The logic $\bL(I)$ adds to first-order logic the quantifier $I$ with the semantics
	\[
	\mathcal A \models I x y \varphi(x) \psi(y) \text{ iff } |\{a \in A \colon \mathcal A \models \varphi(a) \}| = |\{a \in A \colon \mathcal A \models \psi(a) \}|.
	\]
	Second-order logic $\bL^2$ allows quantification over relations on $A$. It is well known, that there is a sentence of second-order logic known as \emph{Magidor's $\Phi$}, axiomatizing the class of structures isomorphic to some $V_\alpha$ for some limit ordinal $\alpha$ (cf. \cite{mag1971}). It is simple to adjust this construction to get a sentence $\Phi^*$ axiomatizing the class of structures isomorphic to some $V_\alpha$ for some arbitrary ordinal $\alpha$.
	
	\emph{Sort logic} introduced by Väänänen \cite{vää1979, vää2014} extends $\bL^2$ by so-called \emph{sort quantifiers} written as $\widetilde{\forall}$ and $\widetilde{\exists}$, which take $m$-ary relation symbols $X$ for some $m \in \omega$. A formula $\widetilde{\exists} X \varphi(X)$ is true in a structure $\mathcal A$ if $\mathcal A$ can be expanded by an additional domain $B$ and a set $Y \subseteq B^m$ such that the expanded structure satisfies the formula $\varphi(B)$, i.e., sort quantifiers search outside the structure itself, looking for any set in $V$ satisfying some relation described by $\varphi$. Sort logic is graded by natural numbers $n$ into the logics $\bL^{s,n}$, which restrict to $n$ alternations of $\widetilde \exists$ and $\widetilde \forall$, as allowing arbitrary alternations would run into definability of truth issues. Writing $C^{(n)}$ to denote the club class of ordinals $\alpha$ such that $V_\alpha$ is a $\Sigma_n$-elementary substructure of the universe $V$, there is a sentence $\Phi^{(n)} \in \bL^{s,n}$ axiomatizing the class of all structures isomorphic to some $V_\alpha$ for $\alpha \in C^{(n)}$ (cf., e.g., \cite[Corollary 1.2.17]{osinski2024}).\footnote{This seems to be a folklore result, implicit, e.g., in \cite{vää2014}. Initial explicit proofs used a (countable) \emph{theory} in $\bL^{s, n}$ (cf., e.g., \cite[Lemma 4.10]{bon2020}), but the second author's thesis shows how to do it in a single sentence.}
	
	Given regular cardinals $\kappa \geq \lambda$, any of these logics can be expanded to infinitary versions by adding conjunctions and disjunctions of size $< \kappa$ and first-order quantifiers of size $< \lambda$. In the case of second-order and sort logics, the infinitary versions $\bL^2_{\kappa, \lambda}$ and $\bL^{s,n}_{\kappa, \lambda}$ additionally come with second-order and sort quantifiers of size $< \lambda$.
	
	Abstract logics allow the generalization of model-theoretic concepts to arbitrary logics. A theory $T \subseteq \cL$, i.e., a set of $\cL$-sentences, is called ${<}\kappa$-satisfiable for a cardinal $\kappa$, if every $T_0 \in \mathcal P_\kappa T$ is satisfiable. A cardinal $\kappa$ is called the \emph{compactness number} $\comp(\cL)$ of a logic $\cL$ if it is the \emph{smallest} cardinal such that every ${<}\kappa$-satisfiable $\cL$-theory is satisfiable. It is called the \emph{Löwenheim-Skolem-Tarski} number $\LST(\cL)$ if it is the smallest cardinal $\kappa$ such that if $\tau$ is a vocabulary of size $< \kappa$ and $\varphi \in \cL[\tau]$, then any $\tau$-structure $\mathcal A \models \varphi$ has a substructure $\mathcal B$ of size $|B| < \kappa$ and satisfying $\varphi$.\footnote{Note that sometimes the $\LST$ number is defined demanding the existence of full $\cL$-elementary substructures, as opposed to substructures preserving satisfaction of single sentences. For many logics it is known that the two definitions are equivalent. For other logics, the notion of `elementary substructure' can be tricky to define, so we focus on ordinary substructure.}
	
	\subsection{Large cardinals}\label{sec:LCs}
	Let us state the large cardinal notions we will consider and some related results.
	
	For a natural number $n$ and an ordinal $\lambda$, a cardinal $\kappa$ is called \emph{$\lambda$-$\Pi_n$-strong} if for any $\Pi_n$-definable (without parameters) class $A$, there is an elementary embedding $j: V \rightarrow M$ such that $\crit(j) = \kappa$, $j(\kappa) > \lambda$, $V_\lambda  \subseteq M$, and $A \cap V_\lambda \subseteq A^M$. It is \emph{$\Pi_n$-strong} if it is $\lambda$-$\Pi_n$-strong for every ordinal $\lambda$. The $\Pi_1$-strong cardinals are precisely the more familiar strong cardinals. Moreover, with rising $n$, the assertion that there is a $\Pi_n$-strong cardinal gains consistency strength. The definition and the mentioned results can be found in \cite{bagaria2023weak}. 
	
	For a definable (with parameters) class $A$ and an ordinal $\lambda$, $\kappa$ is called \emph{$\lambda$-$A$-strong} if there is an elementary embedding $j: V \rightarrow M$ such that $\crit(j) = \kappa$, $j(\kappa) > \lambda$, $V_\lambda \subseteq M$, and $A \cap V_\lambda = A^M \cap V_\lambda$, where $A^M = \{x \in M \colon M \models x \in A\}$. It is $A$-strong if it is $\lambda$-$A$-strong for every $\lambda$.
	
	By ``\emph{Ord is Woodin}", we denote the axiom schema stating that for every class $A$, there is an $A$-strong cardinal.
	
	A cardinal $\kappa$ is \emph{supercompact} if for any $\lambda > \kappa$ there is an elementary embedding $j: V \rightarrow M$ with $\crit(j) = \kappa$, $j(\kappa) > \lambda$, and $M^\lambda \subseteq M$.
	
	A cardinal $\kappa$ is called $C^{(n)}$-extendible if for every $\alpha > \kappa$, there is an elementary embedding $j: V_\alpha \rightarrow V_\beta$ such that $\crit(j) = \kappa$, $j(\kappa) > \alpha$, and $j(\kappa) \in C^{(n)}$ (cf. \cite{bag2012}). It is known that $\kappa$ is $C^{(n)}$-extendible iff for every $\alpha > \kappa$ such that $\alpha \in C^{(n)}$ there is a $\beta \in C^{(n)}$ and an elementary embedding $j: V_\alpha \rightarrow V_\beta$ such that $\crit(j) = \kappa$ (cf. \cite{gitman2019CnExt}). The $C^{(1)}$-extendible cardinals are precisely the more familiar extendible cardinals (cf. \cite{bag2012}). 
	
	\emph{Vopěnka's Principle} (VP) is the axiom schema stating that in every definable class $\mathcal C$ of structures in a shared language, there are distinct $\mathcal A, \mathcal B \in \mathcal C$ such that there is an elementary embedding $e: \mathcal A \rightarrow \mathcal B$. We write $\VP(\Pi_n)$ for the restriction of the statement of VP to classes definable by a $\Pi_n$-formula. VP is closely connected to the existence of $C^{(n)}$-extendible cardinals, as well as to model theory of strong logics:
    \begin{fact}[Bagaria {\cite[Corollary 4.15]{bag2012}}]
        For $n \geq 2$, $\VP(\Pi_{n+1})$ holds if and only if there is a $C^{(n)}$-extendible cardinal.
    \end{fact}
	\begin{fact}[Bagaria, Makowsky]\label{fact:BagariaMakowsky}
	The following axiom schemas are equivalent:
	\begin{enumerate}
        \item[(i)] $\VP$.
		\item[(ii)] For every $n$, there is a $C^{(n)}$-extendible cardinal (cf. \cite[Corollary 4.15]{bag2012}).
		\item[(iii)] Every logic has a compactness number (cf. \cite[Theorem 2]{mak1985}).
	\end{enumerate}
	\end{fact}
Note that \cite[Proposition 4.22]{boney2024model} gives a more direct equivalence between these clauses.
    
	Moreover, Bagaria showed that with rising $n$, the assertion that there is a $C^{(n)}$-extendible cardinal gains consistency strength (cf. \cite{bag2012}). The $C^{(n)}$-extendible cardinals thus form a \emph{pattern} (cf. \cite{bagaria2024patterns}) of stronger and stronger large cardinals below VP. The first author showed that this pattern is mirrored on the model-theoretic side:
	\begin{fact}[Boney {\cite[Proposition 4.12]{bon2020}}]\label{fact:BoneySortLogic}
		For every $n \geq 2$, a cardinal $\kappa$ is the compactness number of $\bL^{s,n}$ if and only if $\kappa$ is the smallest $C^{(n)}$-extendible cardinal.
	\end{fact}
	A cardinal $\kappa$ is \emph{huge with target $\lambda$} if there is an elementary embedding $j: V \rightarrow M$ such that $\crit(j) = \kappa$, $j(\kappa) = \lambda$, and $M^\lambda \subseteq M$.

	\section{Henkin-compactness, $\Pi_n$-strong cardinals, and weak Vopěnka's Principle}\label{sec:WVP}
	In \cite{adamek1988}, the authors considered \emph{weak Vopěnka's Principle} (WVP), an axiom arising from dualizing a category-theoretic formulation of VP. They showed that VP implies WVP, but whether the other implication holds was a long standing open question until recently Wilson answered it negatively (cf. \cite{wilson2020weak,wilson2022large}). Instead, Wilson showed that WVP is equivalent to Ord is Woodin, and thus weaker than, e.g., the existence of a Woodin cardinal, which in turn is much weaker than VP. Moreover, Bagaria and Wilson showed that the $\Pi_n$-strong cardinals form a pattern analogous to the pattern of $C^{(n)}$-extendible cardinals below VP. Let us write $\WVP(\Pi_n)$ for the statement of WVP restricted to $\Pi_n$ definable classes, as was considered in \cite[Definition 2.2]{bagaria2023weak}.
    \begin{fact}[Bagaria \& Wilson {\cite[Theorem 5.13]{bagaria2023weak}}]\label{fact:BagariaWilson0}
		The following are equivalent for $n\geq1$:
		\begin{enumerate}
			\item[(1)] $\WVP(\Pi_n)$.
			\item[(2)] There is a $\Pi_n$-strong cardinal.
		\end{enumerate}
	\end{fact} 
	\begin{fact}[Bagaria \& Wilson \cite{bagaria2023weak}]\label{fact:BagariaWilson}
		The following are equivalent:
		\begin{enumerate}
			\item[(1)] WVP.
			\item[(2)] For every $n$, there is a $\Pi_n$-strong cardinal.
			\item[(3)] For every $n$, there is a proper class of $\Pi_n$-strong cardinals.
            \item[(4)] Ord is Woodin.
		\end{enumerate}
	\end{fact} 
	Given VP's connections to model theory, this left open whether WVP similarly has characterizations by model-theoretic properties of strong logics. We will show that this is indeed the case.\footnote{Note that independently Holy, Lücke, and Müller \cite[Theorem 1.8(2)]{holy2024} also provide a positive answer to this question by entirely different methods than the ones used below, considering what they call \emph{outward compactness cardinals}. Note that they provide a characterization of the global schema Ord is Woodin, without going through local forms in terms of $\Pi_n$-strong cardinals. In contrast to our results, they thus leave open whether the hierarchy of $\Pi_n$-strong cardinals can be characterized by outward compactness properties. A third way of characterizing both $\Pi_n$-strong cardinals and WVP by model-theoretic properties is due to the second author and Wilson (cf. \cite[Section 5.4]{osinski2024}), using downward Löwenheim-Skolem properties of class-sized versions of sort logic.} More precisely, we will show that the exact pattern of $\Pi_n$-strong cardinals below WVP corresponds to certain compactness properties of sort logic. The relation of model theory, WVP, and $\Pi_n$-strong cardinals is thus in some sense analogous to the relation of model theory, VP, and $C^{(n)}$-extendible cardinals described by the first author's Theorem \ref{fact:BoneySortLogic} and the results of Section \ref{sec:ULST} below. We will consider compactness properties related to \emph{Henkin models for an abstract logic}, which were introduced in \cite{boney2024model}.
	
	Recall that, classically, a Henkin model for a sentence $\varphi$ of second-order logic consists of a structure $\mathcal A$ and a subset $P \subseteq \mathcal P(A)$ such that $(\mathcal A, P)$ computes to be a model of $\varphi$ when letting the second-order quantifiers in $\varphi$ range over $P$, as opposed to the full powerset of $A$. Note that if $(M, \in)$ is some transitive model of set theory such that 
	\[
	(M, \in) \models ``\mathcal A \models_{\bL^2} \varphi",
	\]
	then $(\mathcal A, \mathcal P^M(A))$ is a Henkin model of $\varphi$. In \cite{boney2024model}, this observation inspired the authors to introduce the notion of a \textit{Henkin model for an abstract logic}. The following is a slightly less general version of this notion.
	\begin{definition}
		Let $\mathcal L$ be a logic and $T \subseteq \mathcal L[\tau]$ a theory. A pair $\mathcal M = (M, \cA)$ is an \emph{$\mathcal L$-Henkin model of $T$} iff $M$ includes all possible parameters used in the definition of $\cL$ and
		\begin{enumerate}
			\item[(1)] $(M, \in)$ is a transitive model of some large finite fragment of ZFC, and
			\item[(2)] $\mathcal A \in M$ and there is a renaming of $\tau$ that induces a copy $T^* \subset \cM$ of $T$ such that $\cA$ is a $\sigma$-structure for some $\sigma \supseteq \tau(T^*)$, and for every $\varphi \in T^*$, we have $\mathcal M \models `` \mathcal A \models_{\mathcal L} \varphi"$. 
		\end{enumerate}
	\end{definition}
	
	As in \cite{bon2020}, we say that an $\mathcal L$-Henkin model $\mathcal M$ is \textit{full up to rank $\lambda$} for some ordinal $\lambda$ if $V_\lambda \subseteq M$ and it is \textit{$n$-correct for a set $A$} iff for every $\Sigma_n$-formula $\Phi(x)$ and every $a \in M \cap A$, we have $\mathcal M \models \Phi(a)$ iff $V \models \Phi(a)$. For every natural number $n$, fix a finite fragment $\ZFC^*_n$ which is large enough to show that $\Phi^*$ axiomatizes the class of models isomorphic to some $V_\alpha$ and that $\Phi^{(n)}$ axiomatizes the class of models isomorphic to some $V_\alpha$ with $\alpha \in C^{(n)}$. With this we show the following theorem:
	
	\begin{theorem}\label{thm:PinStrong}
		The following are equivalent for every cardinal $\kappa$ and $n \geq 2$:
		\begin{enumerate}
			\item[(1)] $\kappa$ is $\Pi_n$-strong
			\item[(2)] For every $\lambda \in C^{(n)}$ and every theory $T \subseteq \bL_{\kappa, \omega}^{s, n}$ that can be written as an increasing union $T = \bigcup_{\alpha < \kappa} T_\alpha$ of theories $T_\alpha$ that each have models of size $\geq \kappa$, there is an $\bL_{\kappa, \omega}^{s, n}$-Henkin model $\mathcal M = (M, \mathcal A)$ of $T$ such that
			\begin{enumerate}
				\item $\mathcal M \models \ZFC^*_n$,
				\item $|A| \geq \lambda$, and
				\item $\mathcal M$ is full up to $V_\lambda$ and $n$-correct for $V_\lambda$.
			\end{enumerate}
		\end{enumerate}
		\begin{proof}
            Both clauses are clearly false for $\kappa = \omega$, so we can assume $\kappa$ is uncountable (this allows us to take countable conjunctions below).
        
			For the forward direction, note that it sufficient to show (2) for arbitrarily large $\lambda$, so we can assume we are given $\lambda$ which is a limit point of $C^{(n)}$. Suppose we have a setup like in (2), i.e., a theory $T \subseteq \bL_{\kappa, \omega}^{s,n}$ and an increasing union $\bigcup_{\alpha < \kappa} T_\alpha = T$ such that every $T_\alpha$ has a model of size $\geq \kappa$. Then we can pick a function $f$ on $\kappa$ such that every $f(\alpha) = (V_{\beta_\alpha}, \mathcal A_\alpha)$ is an $\bL^{s,n}_{\kappa, \omega}$-Henkin model of $T_\alpha$. Without loss of generality we can chose $\kappa \leq \beta_\alpha \in C^{(n)}$ such that $V_{\beta_\alpha} \models \ZFC^*_n$. 
			
			Now as $\kappa$ is $\lambda$-$\Pi_n$-strong and $\lambda$ a limit of $C^{(n)}$, by \cite[Proposition 5.9]{bagaria2023weak} let $j: V \rightarrow N$ be elementary such that $\crit(j) = \kappa$, $j(\kappa) > \lambda$, $V_\lambda \subseteq N$ and $N \models \lambda \in C^{(n)}$. Then $j(f)$ is a function on $j(\kappa)$ such that, in $N$, $j(f)(\alpha)$ is an $\bL^{s,n}_{j(\kappa), \omega}$-Henkin model of $T_\alpha^*$, where $j(T) = \bigcup_{\alpha < j(\kappa)} T_\alpha^*$. Thus, in $N$, we have $j(f)(\kappa) = (M, \mathcal A)$ and is an $\bL^{s,n}_{j(\kappa), \omega}$-Henkin model of $T_\kappa^* \supseteq \bigcup_{\alpha < \kappa} T_\alpha^* = \bigcup_{\alpha < \kappa} j(T_\alpha) \supseteq j``T$. Given standard syntax of sort logic (cf., e.g., \cite{vää2014}) it is easy to see that $j``T \subseteq \bL^{s,n}_{\kappa, \omega}$ is a copy of $T$. Thus, from the outside, we see that $(M, \mathcal A)$ is an $\bL_{\kappa, \omega}^{s,n}$-Henkin model of $T$. 
			
			As all $f(\alpha)$ are full up to rank $\kappa$, $N$ believes that $j(f)(\kappa)$ is full up to rank $j(\kappa) > \lambda$. $N$ might be incorrect about this, but as $V_\lambda \subseteq N$ we at least have that $j(f)(\kappa)$ is full up to rank $\lambda$. To show that $j(f)(\kappa)$ is $n$-correct for $V_\lambda$ let $a \in V_\lambda$ and let $\Phi(x)$ be $\Sigma_n$. Then 
			\begin{equation*}
				\begin{split}
					V \models \Phi(a) & \Leftrightarrow V_\lambda \models \Phi(a) \\
					& \Leftrightarrow N \models \Phi(a) \\
					& \Leftrightarrow N \models ``M \models \Phi(a)" \\
					& \Leftrightarrow M \models \Phi(a).
				\end{split}
			\end{equation*}
			The first equivalence holds because $\lambda \in C^{(n)}$, the second one because $N \models \lambda \in C^{(n)}$, the third because, as all the $\beta_\alpha \in C^{(n)}$, $N \models ``M = V_\beta \text{ for some } \beta \in C^{(n)}"$, and the last by absoluteness of first-order satisfaction. So $j(f)(\kappa)$ is an $\bL^{s,n}_{\kappa, \omega}$-Henkin model of $j``T$ with our desired properties.
			
			\medskip For the backward direction, suppose we have the compactness property from (2). By Lemma \ref{lem:localPinStrong} below, given $\lambda = \beth_\lambda$ a limit of $C^{(n)}$, it is sufficient to provide an elementary embedding to a well-founded model $j: V_{\kappa + 1} \rightarrow M$ with $\crit(j) = \kappa$, $j(\kappa) \geq \lambda$, $V_\lambda \subseteq M$ and $C^{(n)} \cap V_\lambda = j(C^{(n)} \cap V_\kappa) \cap V_\lambda$. Consider the following theory: 
			\begin{equation*}
				\begin{split}
					T = & \text{ED}_{\bL_{\kappa, \omega}}(V_{\kappa + 1}, \in) \cup \{c_i < c < c_\kappa \colon i < \kappa\} \cup \{\Phi\} \cup \\
					& \{\forall x(x \in c_{C^{(n)} \cap V_\kappa} \rightarrow (\Phi^{(n)})^{\{y\colon y \in V_x\}} \} \cup
					\{\forall x(( \Phi^{(n)})^{\{y \colon y \in V_x\}} \wedge x \in c_{V_\lambda} \rightarrow x \in c_{C^{(n)} \cap V_\kappa} \},
				\end{split}
			\end{equation*}
			where $(\Phi^{(n)})^{\{y \colon y \in x\}}$ is the relativization of $\Phi^{(n)}$ to the substructure which consists of the elements the structure believes to be a member of $V_x$, i.e., a formula in $\bL^{s, n}$ coding that $x \in C^{(n)}$. Clearly, this theory can be written as an increasing union of satisfiable theories $T_\alpha$ for $\alpha < \kappa$ by considering those bits of $T$ that include only the sentences $c_i < c < c_\kappa$ for $i < \alpha$ and using $V_{\kappa + 1}$ as a model. Then by (2), $T$ has an $\bL^{s,n}_{\kappa, \omega}$-Henkin model $(M, N)$ satisfying $\ZFC^*_n$ which is full up to rank $\lambda$, $n$-correct for $V_\lambda$, and with $|N| \geq \lambda$. Because $M$ believes that $N$ satisfies $\Phi$, $N = V_\beta^M$ for some $\beta$. In particular, $N$ is well-founded. From the outside we see that $N$ satisfies a copy of the $\bL_{\kappa ,\omega}$-elementary diagram of $V_{\kappa + 1}$, so there is an elementary embedding $j: V_{\kappa + 1} \rightarrow N$. Since $\bL_{\kappa, \omega}$ can define all ordinals $< \kappa$, it follows that $\crit(j) \geq \kappa$. By the sentences $c_i < c < c_\kappa$, we get $\crit(j) \leq \kappa$, and so $\crit(j) = \kappa$. Because $\lambda$ is a $\beth$-fixed point and $|V_\beta^M|\geq \lambda$, we have to have $\beta \geq \lambda$. Because $V_\lambda \subseteq M$, therefore $V_\lambda \subseteq V_\beta^M$. Note that $\kappa$ is the largest cardinal of $V_{\kappa + 1}$, so $j(\kappa)$ is the largest $M$-cardinal. In particular, $j(\kappa) \geq \lambda$.

			Finally, for $\alpha < \lambda$, we have
			\[
			M \models ``N \models \alpha \in c_{C^{(n)} \cap V_\lambda}" \text{ iff } M \models `` N \models (\wedge \Phi^{(n)})^{\{y \colon y \in V_\alpha\}}".
			\]
			The first part means that $\alpha \in j(C^{(n)} \cap V_\lambda)$, while the second is equivalent to $M \models \alpha \in C^{(n)}$. Because $M$ is $n$-correct for $V_\lambda$ and $C^{(n)}$ is $\Pi_n$ definable, for $\alpha < \lambda$ the latter actually means $\alpha \in C^{(n)}$. Thus $\alpha \in j(C^{(n)} \cap V_\lambda) \cap V_\lambda$ iff $\alpha \in C^{(n)} \cap V_\lambda$. 
			
		\end{proof}
	\end{theorem}
	
	\begin{lemma}\label{lem:localPinStrong}
		If for arbitrarily large limits $\lambda$ of $C^{(n)}$, there are elementary embeddings $j: V_{\kappa + 1} \rightarrow M$ for a transitive $M$, where $\crit(j) = \kappa$, $j(\kappa) \geq \lambda$, $V_\lambda \subseteq M$ and $j(C^{(n)} \cap V_\kappa) \cap V_\lambda = C^{(n)} \cap V_\lambda$, then $\kappa$ is $\Pi_n$-strong.
		\begin{proof}
			Let $\lambda$ be a limit point of $C^{(n)}$. Note that it suffices to show that $\kappa$ is $\lambda$-$\Pi_n$-strong. Note further that analogously to the case for strong cardinals (cf., e.g., \cite[Exercise 26.7(b)]{kan}), the condition that $j(\kappa) > \lambda$ is superfluous for the global notion of being $\Pi_n$-strong. Thus, by \cite[Proposition 5.9]{bagaria2023weak}, it is sufficient to find an elementary embedding $j: V \rightarrow N$ such that $\crit(j) = \kappa$, $V_\lambda \subseteq N$, and $N \models \lambda \in C^{(n)}$. Our assumption provides an elementary embedding $j: V_{\kappa + 1} \rightarrow M$ where $\crit(j) = \kappa$, $j(\kappa) \geq \lambda$, $V_\lambda \subseteq M$ and $j(C^{(n)} \cap V_\kappa) \cap V_\lambda = C^{(n)} \cap V_\lambda$. We may derive a $(\kappa, \lambda)$-extender $E$ from $j$ in the standard way by letting for $a \in [\lambda]^{< \omega}$ and $X \subseteq [\kappa]^{< \omega}$:
			\[
			X \in E_a \text{ iff } a \in j(X).
			\]
			Proceed to build the extender power $m_E$ of $V_{\kappa + 1}$ by $E$. Standard arguments (analogous to those in \cite[Section 26]{kan}) imply that this comes with elementary maps $i_E: V_{\kappa + 1} \rightarrow m_E$ and $k_E: m_E \rightarrow M$ such that $\crit(i_E) = \kappa$, $i_E(\kappa) \geq \lambda$, $V_\lambda \subseteq m_E$, $j = k_E \circ i_E$, and $k_E \upharpoonright \lambda = \id$. Further, $m_E$ is an initial segment of the extender power $M_E$ of $V$ and the standard elementary map $j_E: V \rightarrow M$ restricts to $i_E$. In particular, $\crit(j_E) = \kappa$ and $V_\lambda \subseteq M_E$. It is sufficient to show that $j_E(C^{(n)} \cap V_\kappa) \cap V_\lambda = C^{(n)} \cap V_\lambda$, as then $M_E$ will see that $C^{(n)}$ is unbounded below $\lambda$, and therefore that $\lambda \in C^{(n)}$. Because $j_E$ restricts to $i_E$, it is thus sufficient to show the following claim.
			\begin{claim} 
				$i_E(C^{(n)}\cap V_\kappa) \cap V_\lambda =  C^{(n)} \cap V_\lambda$.
			\end{claim}
			By elementarity and because $k_E \circ i_E = j$, we get that
			\[
			\alpha \in i_E(C^{(n)} \cap V_\kappa) \text{ iff } k_E(\alpha) \in k_E(i_E(C^{(n)} \cap V_\kappa)) = j(C^{(n)} \cap V_\kappa).
			\]
			Now $k_E \upharpoonright \lambda = \id$, so if $\alpha < \lambda$, then $k_E(\alpha) = \alpha$. Together we have:
			\[
			\alpha \in i_E(C^{(n)} \cap V_\kappa) \cap V_\lambda \text{ iff } \alpha \in j(C^{(n)} \cap V_\kappa) \cap V_\lambda = C^{(n)} \cap V_\lambda.
			\]
		\end{proof}  
	\end{lemma}
    It is easy to adapt the proof of Theorem \ref{thm:PinStrong} to characterize the smallest $\Pi_n$-strong cardinal (as was explicitly carried out in \cite[Theorem 2.3.7]{osinski2024}). Then we get:
    \begin{theorem}\label{thm:HCCsort}
        The following are equivalent for every cardinal $\kappa$ and $n \geq 2$:
        \begin{enumerate}
			\item[(1)] $\kappa$ is the smallest $\Pi_n$-strong cardinal. 
			\item[(2)] $\kappa$ is the smallest cardinal such that for every $\lambda \in C^{(n)}$ and every theory $T \subseteq \bL^{s, n}$ that can be written as an increasing union $T = \bigcup_{\alpha < \kappa} T_\alpha$ of theories $T_\alpha$ that each have models of size $\geq \kappa$, there is an $\bL^{s, n}$-Henkin model $\mathcal M = (M, \mathcal A)$ of $T$ such that
			\begin{enumerate}
				\item $\mathcal M \models \ZFC^*_n$,
				\item $|A| \geq \lambda$, and
				\item $\mathcal M$ is full up to $V_\lambda$ and $n$-correct for $V_\lambda$.
			\end{enumerate}
		\end{enumerate}
    \end{theorem}
    For a logic $\cL$, let us call the cardinal $\kappa$ that we get from exchanging $\bL^{s,n}$ in (2) above for $\cL$ the \emph{$n$-Henkin chain compactness} (HCC) number of $\cL$ and write $n\text{-HCC}(\cL)=\kappa$. Combined with Fact \ref{fact:BagariaWilson0}, Theorem \ref{thm:HCCsort} above then shows that for $n\geq 2$, the existence $n\text{-HCC}(\bL^{s,n})$ is equivalent to $\WVP({\Pi_n})$. Adapting the first author's argument for \cite[Theorem 4.7]{bon2020} to the setting of Henkin models considered here shows further that $1\text{-HCC}(\bL^2)$ is the smallest strong cardinal. The existence of $1\text{-HCC}(\cL^2)$ is thus equivalent to $\WVP(\Pi_1)$.

    Moreover, because the strength of any logic is bounded\footnote{A proof of this probably folklore result can be found in \cite[Corollary 1.2.24]{osinski2024}.} by $\bL^{s,n}_{\kappa, \omega}$ for some $n$ and $\kappa$, combining Theorems \ref{fact:BagariaWilson} and \ref{thm:PinStrong}, we get a model-theoretic characterization of WVP and ``Ord is Woodin" reminiscent of Makowsky's result about VP (cf.\ Theorem \ref{fact:BagariaMakowsky}). Our results show that the characterization also localizes to $\Pi_n$-strong cardinals analogously to how Makowsky's result localizes to $C^{(n)}$-extendible cardinals by the first author's Theorem \ref{fact:BoneySortLogic}.
	\begin{corollary}
		The following are equivalent:
		\begin{enumerate}
			\item[(1)] WVP.
            \item[(2)] Ord is Woodin.
			\item[(3)] For every natural number $n$ and every logic $\cL$, $n{\text{-HCC}(\cL)}$ exists.
		\end{enumerate}
	\end{corollary}

    Figure 1 illustrates the results connecting WVP, $\Pi_n$-strong cardinals and the Henkin chain compactness properties isolated here.

    \begin{figure}[h]
	\centering
	\noindent \includegraphics[width=0.99\hsize]{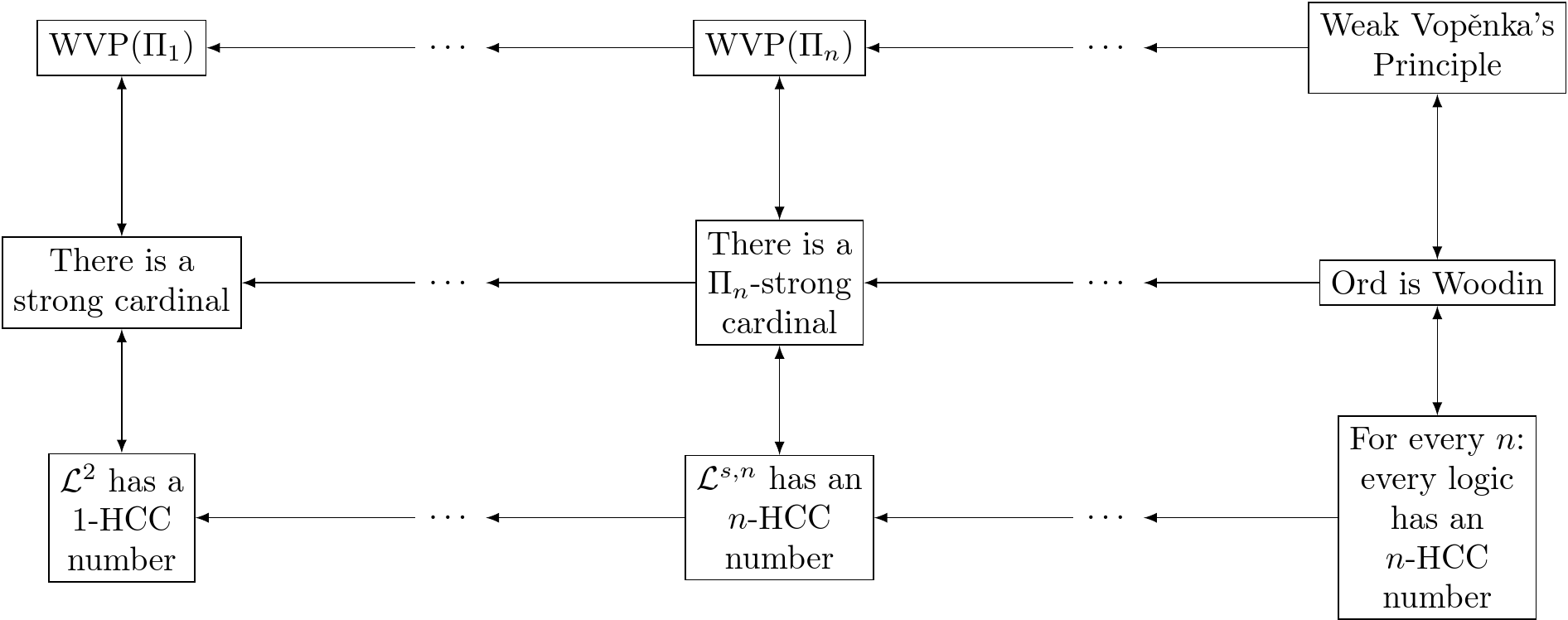}
	\caption{Relations between $\WVP$, $\Pi_n$-strong cardinals, and HCC numbers.}
    \end{figure}
    
	The analogy between VP and WVP can be pushed even further. In \cite{osinski2024Henkin}, the second author and Poveda show that also $C^{(n)}$-extendible cardinals, and hence VP, can be characterized by compactness properties with Henkin models by considering a stronger notion of Henkin model and switching to theories $T$ which are ${<} \kappa$-satisfiable, as opposed to satisfiable along a $\kappa$-chain as above. The patterns of $C^{(n)}$-extendible cardinals below VP, and of $\Pi_n$-strong cardinals below WVP can thus be obtained from each other by switching between HCC numbers and the compactness properties for the stronger notion of Henkin models from \cite{osinski2024Henkin}.  
	
	We would like to make some more comments about related results. The first author \cite[Theorem 4.7]{bon2020} showed that strong cardinals can be characterized via a compactness property that infers the existence of a classical Henkin model for second-order logic of a theory $T \subseteq \bL^2$ which is satisfiable along a $\kappa$-chain as the theory considered in the compactness principle from Theorem \ref{thm:PinStrong}. And in \cite[Theorem 3.7]{boney2024model}, the authors provide a characterization of Woodin cardinals by compactness properties providing the existence of Henkin models. Our theorem \ref{thm:PinStrong} can both be seen as a generalization of the characterization of strength to $\Pi_n$-strength, and as a localization of the characterization of Woodinness. 
	
	\section{Huge type omission in finitary logic}\label{sec:huge}
	
	Makowsky's and Stavi's results about VP motivated the question whether there are axioms about model-theoretic properties of strong logics with higher consistency strength than VP. This is indeed the case. The first author showed that certain \emph{compactness for type omission} properties of $\bL_{\kappa, \kappa}$ imply the existence of a huge cardinal (cf. \cite[Corollary 3.7(b)]{bon2020}). Further, Wilson showed that huge cardinals can be characterized by Löwenheim-Skolem properties of fragments of the class-sized logic $\bL_{\infty, \infty}$. Wilson's results remain unpublished but were presented, e.g., in a talk at the European Set Theory Conference 2022 \cite{wil2022b}. In the very same talk, Wilson asked whether considering properties of \emph{infinitary} logics -- as opposed to finitary ones -- is necessary to get large cardinal strength above VP by model-theoretic means, as the first author's and his results might suggest. We will answer this question negatively by showing that compactness for type omission of $\bL(\WF)$ can characterize huge cardinals.
	
	Compactness for type omission principles provide models of a theory such that the model simultaneously omits a given type. Benda showed that supercompact cardinals can be characterized in this way (cf. \cite{ben1978}). The first author expanded on this result and introduced a general framework given in the following definition to formulate compactness for type omission properties to show that many large cardinals can be characterized in that way (cf. \cite{bon2020}).

    One important observation is that, in contrast to the classic compactness principles, $\omega$ does not satisfy the supercompact-style compactness for type omission with $\bL_{\omega, \omega}$.  This is related to the idea that it is $\ZFC$-provable that there is no fine ultrafilter on $\cP_\omega \lambda$ such that every regressive function is constant on a $U$-large set; this is witnessed by the regressive function $\max$.  Thus, the attempt to find compactness for type omission in first-order logic must look at cardinals above $\omega$.
	\begin{definition}
		Let $\kappa$ be a cardinal, $\kappa \leq \lambda$, $I \subseteq \mathcal P(\lambda)$, and $\cL$ a logic.
		\begin{enumerate}
			\item[(i)] $I$ is \emph{$\kappa$-robust} if for every $\alpha < \kappa$, we have $I \subseteq \{s \in \mathcal P(\lambda) \colon |s \cap \kappa| < \kappa\}$ and $\{s \in I \colon \alpha \in s\} \neq \emptyset$.
			\item[(ii)] $C \subseteq I$ \emph{contains a strong $\kappa$-club} if there is a function $F: [\lambda]^2 \rightarrow \mathcal P_\kappa \lambda$ such that 
			\[
			C(F) = \{s \in I \colon |s|\geq \omega \wedge \forall x,y \in s (F(x,y) \subseteq s) \}\subseteq C.
			\]
			\item[(iii)] Let $X$ be a set that is written as an increasing union $X = \bigcup_{s \in I} X_s$, i.e., for $s \subseteq t$, also $X_s \subseteq X_t$. We say that \emph{the union respects the index} iff there is a collection $\{X^\alpha \colon \alpha \in \lambda\}$ such that for each $s \in I$:
			\[
			X_s = \bigcup_{\alpha \in s} X^\alpha.
			\]
			\item[(iv)] Given a $\kappa$-robust $I$, the logic $\cL$ is \emph{$I$-$\kappa$-compact for type omission} if the following holds: For any theory $T \subseteq \cL$ which can be written as an increasing union $T = \bigcup_{s \in I} T_s$ that respects the index, and any type $p(x) = \{\varphi_i(x) \colon i < \lambda\}$ with subsets $p_s = \{\varphi_i(x) \colon i \in s\} \subseteq \cL$ for $s \in I$, if the set 
			\[
			\{s \in I \colon T_s \text{ has a model omitting }p_s \}
			\]
			contains a strong $\kappa$-club, then $T$ has a model omitting $p$.
		\end{enumerate}
	\end{definition}
	A result by Menas shows that containing a strong $\kappa$-club can be considered as a generalization of being a member of the club filter over $\mathcal P_\kappa \lambda$ to the case of arbitrary $I \subseteq \mathcal P(\lambda)$ (cf., e.g., \cite[Proposition 25.3]{kan}). Recall that any fine, normal, $\kappa$-complete ultrafilter over $\mathcal P_\kappa \lambda$ contains the club filter (cf., e.g., \cite[Proposition 25.4]{kan}). The first author showed that the notion of containing a strong $\kappa$-club extends this result in the following way:
	\begin{fact}[Boney {\cite[Fact 3.2]{bon2020}}]
		Let $I \subseteq \mathcal P(\lambda)$ be $\kappa$-robust. If $U$ is a fine, normal, $\kappa$-complete ultrafilter over $I$, then $C(F) \in U$ for all $F: [\lambda]^2 \rightarrow \mathcal P_\kappa \lambda$. 
	\end{fact}  
	Thus, the intuition for $I$-$\kappa$ compactness for type omission is, that if on a large set of $s \in I$, $T_s$ has a model omitting $p_s$, then $T$ has a model omitting $p$, where by ``large set" we mean one contained in every possible fine, normal, $\kappa$-complete ultrafilter over $I$.
	
	Recall that $\kappa$ is huge with target $\lambda$ if and only if there is a fine, normal, $\kappa$-complete ultrafilter $U$ over 
	\[
	[\lambda]^\kappa_* = \{s \in \mathcal P(\lambda) \colon \ot(s) = \kappa \text{ and } s \setminus \kappa \neq \emptyset \}.
	\]
	Note that usually, the condition ``$s \setminus \kappa \neq \emptyset$" is not included (cf., e.g., \cite[Theorem 24.8]{kan}), but it is simple to see that it can be adjoined. The reason we include it is that without it, our index set would not be $\kappa$-robust.
	
	The first author showed the following theorem.
	\begin{fact}[Boney {\cite[Corollary 3.7(3)]{bon2020}}]\label{fact:BoneyHuge}
		A cardinal $\kappa$ is huge with target $\lambda$ iff $\bL_{\kappa, \kappa}$ is $[\lambda]^\kappa_*$-$\kappa$-compact for type omission.
	\end{fact} 
	We will show that we may replace $\bL_{\kappa, \kappa}$ by $\bL(\WF)$. We will use the following lemma.
	\begin{lemma}
		Let $\kappa$ and $\lambda$ be cardinals such that $\lambda > \kappa$. If for all transitive sets $M$ with $\lambda \in M$ and $M^\kappa \subseteq M$ there is a transitive set $N$ and an elementary embedding $j: M \rightarrow N$ such that $\crit(j) = \kappa$, $\lambda = j(\kappa)$ and $j `` \lambda \in N$, then $\kappa$ is huge with target $\lambda$.
		\begin{proof}
			Take some cardinal $\gamma > \lambda$ such that $V_\gamma$ is closed under $\kappa$-sequences and $V_{\lambda + 1} \in V_\gamma$. Then by our assumption, we have that there is a transitive $N$ and an elementary embedding $j: V_\gamma \rightarrow N$ such that $\crit(j) = \kappa$, $j(\kappa) = \lambda$ and $j `` \lambda \in N$. Now define an ultrafilter $U$ on $[\lambda]^\kappa_*$ by letting for $X \subseteq [\lambda]^\kappa_*$:
			\[
			X \in U \text{ iff } j``\lambda \in j(X).
			\]
			It is standard to check that $U$ is a fine, normal, and $\kappa$-complete ultrafilter.	
		\end{proof}
	\end{lemma}

	\begin{theorem}\label{thm:hugeTypeOmissionWF}
		Let $\kappa$ and $\lambda$ be cardinals such that $\lambda > \kappa$. Then $\kappa$ is huge with target $\lambda$ iff $\bL(\WF)$ is $[\lambda]^\kappa_*$-$\kappa$-compact for type omission.
		\begin{proof}
			The forward direction immediately follows from \ref{fact:BoneyHuge}, because $\bL_{\kappa \kappa}$ can define well-foundedness. So assume that $\bL(\WF)$ is $[\lambda]^\kappa_*$-$\kappa$-compact for type omission and take a transitive $M$ with $\lambda \in M$ and such that $M^\kappa \subseteq M$. We will show that there is an elementary embedding $j: M \rightarrow N$ such that $N$ is transitive, $\crit(j) = \kappa$, $j(\kappa) = \lambda$ and $j``\lambda \in N$. Let $\{c_i \colon i \in M\}$ be the collection of variables used to formulate the elementary diagram of $M$ and let $c$ and $d$ be new constants. Then let
			\[
			T = \text{ElDiag}_{\bL(\WF)}(M) \cup \{c_\alpha < c < c_\kappa \colon \alpha < \kappa \}\cup \{c_\alpha \in d \wedge |d| = c_\kappa \colon \alpha < \lambda \}
			\] 
			and
			\[
			p(x) = \{x \in d \cup c \} \cup \{x \neq c_\alpha \colon \alpha < \lambda \}.
			\]
			Let us argue that if $N$ is a model of $T$ omitting $p$, then $N$ has all required properties by letting $j: M \rightarrow N$, $c_x^M \mapsto c_x^N$. Here we identify $N$ with its transitive collapse (and so $N$ is transitive) as justified by $N$ being well-founded and extensional by $T$.
			
			Because $c^N$ has order type at least $\kappa$ and is smaller than $c_\kappa^N = j(\kappa)$, we have to have $j(\kappa) > \kappa$. In particular, $\crit(j) \leq \kappa$. To see that $\crit(j) = \kappa$, the type $p$ comes into play. Assume that $\alpha = \crit(j) < \kappa$ and consider $\alpha + 1$. Clearly $c_\beta^N = \beta = j(\beta) < \alpha + 1 < c^N$ for all $\beta < \alpha$. But also $j(\alpha) > \alpha + 1$ and so $c_\beta^N > \alpha + 1$ for all $\beta \geq \alpha$. This means that $\alpha + 1$ realizes $p$, which is a contradiction. So $\crit(j) = \kappa$.

			We also have that $d^N = j``\lambda$ and so $j``\lambda \in N$: Because $c_\alpha^N \in d^N$ for all $\alpha < \lambda$ we have $j``\lambda \subseteq d^N$. And now if $x \in d^N$ and $x \neq j(\alpha) = c_\alpha^N$ for every $\alpha < \lambda$, then $x$ would realize $p$. Thus also $d^N \subseteq j``\lambda$.

			Finally, $N$ knows that $d^N = j``\lambda$ is a set of ordinals and thus has an order type. But clearly $\ot(j `` \lambda) = \lambda$, thus $N \models \ot(d^N) = \lambda$ and further $N \models j(\kappa) = c_\kappa = |d| = \lambda$ and thus $j(\kappa) = \lambda$. 
			
			So we are done if we can show that $T$ has a model omitting $p$. For this purpose we show the following claim.
			\begin{claim} 
				For every $s \in [\lambda]^\kappa_*$, if $s \cap \kappa$ is a limit ordinal $< \kappa$, then $M$  can be expanded to a model of $T_s$ omitting $p_s$, where
				\[
				T_s = \text{ElDiag}(M)_{\bL(\WF)} \cup \{c_\alpha < c < c_\kappa \colon \alpha \in \kappa \cap s\} \cup  \{c_\alpha \in d \wedge |d| = c_\kappa \colon \alpha \in s \}
				\]
				and
				\[
				p_s(x) = \{x \in d \cup c\} \cup \{x \neq c_\alpha \colon \alpha \in s \}.
				\]
			\end{claim}
			To show the claim, suppose $s \in [\lambda]^\kappa_*$ such that $s \cap \kappa$ is a limit ordinal $s \cap \kappa = \beta < \kappa$. Let $c^M = \beta$ and $d^M = s$. Note that $s \in M$ by closure of $M$ under $\kappa$ sequences, so the definition of $c^M$ and $d^M$ make sense. Then $(M,c^M,d^M) \models T_s$: Clearly $c^M = \beta < \kappa = c_\kappa^M$. And if $\alpha \in \kappa \cap s = \beta$, then $c_\alpha^M = \alpha \in \beta$.  Further $\alpha = c_\alpha^M \in s = d^M$ for all $\alpha \in s$ and $|d^M| = |s| = \kappa = c_\kappa^M$. Also, we get that $(M, c^M, d^M)$ omits $p_s$. For if $x \in c^M = s \cap \kappa$, then $x = \alpha = c_\alpha^M$ for some $\alpha \in s$. And if $x \in d^M$, then $x = \alpha = c_\alpha^M$ for some $\alpha \in d^M = s$. So any $x$ omits $p$.
			
			Therefore, to get a model of $T$ omitting $p$, by $[\lambda]^\kappa_*$-$\kappa$-type-omission and the claim, it is sufficient to find a function $F: [\lambda]^2 \rightarrow \mathcal P_\kappa \lambda$ such that for any $s \in [\lambda]^\kappa_*$ with the property that for all $x,y \in s$ we have $F(x,y) \subseteq s$, it holds that $s \cap \kappa$ is a limit ordinal $< \kappa$. Such an $F$ is given by
			\[
			F(x,y) = 
			\begin{cases}
				x+2, & \text{ if } x \in \kappa\\
				\emptyset, & \text{ otherwise.}
			\end{cases}
			\]
			If $s \in [\lambda]^\kappa_*$ is given such that for all $x,y \in s$ we have $F(x,y) \subseteq s$. Then if $\alpha \in s \cap \kappa$, by assumption, $F(\alpha, \alpha) = \alpha + 2 \subseteq s$. Thus $\beta \in s \cap \kappa$ for all $\beta \leq \alpha + 1$. This shows that $s \cap \kappa$ is a limit ordinal. By definition of $[\lambda]^\kappa_*$, $\ot(s \cap \kappa) < \kappa$. Hence, $s \cap \kappa$ is a limit ordinal less than $\kappa$.
		\end{proof}
	\end{theorem}
	Let us mention some related results. In \cite[Theorem 12]{hay2022}, Hayut and Magidor show that supercompact cardinals can be characterized by $\mathcal P_\kappa \lambda$-$\kappa$-compactness for type omission of first-order logic. Our result is the analogue for type omission compactness indexed by $[\lambda]^\kappa_*$ in the stronger logic $\bL(Q^{WF})$. We do not know whether the same is true for huge type omission.
	\begin{question}
		If first-order logic is $[\lambda]^\kappa_*$-$\kappa$-compact for type omission, is $\kappa$ huge with target $\lambda$?
	\end{question}

We also provide a simpler proof of the difficult direction of \cite[Theorem 12]{hay2022}\footnote{Note that the parameterization is slightly off, but the global result (cf. \cite[Corollary 15]{hay2022}) remains the same.}.  
    
    \begin{fact}
    Suppose that first-order logic $\bL$ is $\cP_\kappa\lambda$-$\kappa$-compact for type omission\footnote{\cite{hay2022} call this $\kappa-\cL_\kappa,\kappa$ compactness with type omission.}.  Then $\kappa$ is $\lambda$-supercompact.
    \end{fact}
    
    Hayut and Magidor prove this by using the compactness to show that all relevant transitive models $M$ can be embedded into an extension $N$ that would witness the supercompactness {\bf except} that $N$ is not required to be transitive, and then using that statement to prove that you can find a transitive $N$ \cite[Claim 13]{hay2022}.  We bypass this by building a normal fine ultrafilter on $\cP_\kappa \lambda$, and then such an ultrafilter must be $\kappa$-complete.

    \begin{proof} We follow the standard argument to model-theoretically built a normal ultrafilter (cf., e.g., \cite[Theorem 3.5, (1) implies (3)]{bon2020}).  Fix the structure and language
    \begin{eqnarray*}
    M &=& \langle\cP(\cP_\kappa\lambda),  \cP_\kappa\lambda, \in, X \rangle_{X \subset \cP_\kappa\lambda}\\
    \tau &=& \langle P, Q, E, c_X, d\rangle_{X \subset \cP_\kappa \lambda}
    \end{eqnarray*}
    where $M$ is a $(\tau-\{d\})$-structure.  Then we can define the following theories, sets, and types given $F: \cP_\kappa\lambda \to \lambda$ (recall $[t] = \{s \in \cP_\kappa \lambda : t \subset s\}$ for $t \in \cP_\kappa\lambda$):
    \begin{eqnarray*}
    T^0&:=& Th_\bL(M)\\
    T^1 &:=& \{\text{`}d E c_{[t]}\text{'} : t \in \cP_\kappa \lambda\}\\
    T_s &:=& T^0 \cup \{\text{`}d E c_{[t]}\text{'} : t \subset s\}\\
    X_F &:=& \{s \in \cP_{\kappa} \lambda: F(s) \in s\}\\
    X_{F, \alpha} &:=& \{s\in \cP_\kappa\lambda : F(s) = \alpha\}\\
    p_F(x) &:=& \left\{ \text{`} x = d \wedge x E c_{X_F} \wedge \neg (x E c_{X_{F, \alpha}}) \text{'}: \alpha < \lambda  \right\}
    \end{eqnarray*}
    From this, for each $s \subset \lambda$, we have
    \begin{eqnarray*}
        T^0\cup T^1 &=& \bigcup_{s \in \cP_\kappa\lambda} T_s\\
        \text{`$x \in N$ omits $p_{F, s}(x)$'} & \iff &  N \vDash \text{`}d E c_{X_F} \to \bigvee_{\alpha \in s}d E c_{X_{F, \alpha}}\text{'} 
    \end{eqnarray*}
    This set up gives us an opportunity to apply our compactness for type omission.  For each $s \in \cP_\kappa \lambda$, we wish to find a model $N_s$ of $T_s$ that omits $p_{F, s}$.  As in \cite[Theorem 3.5]{bon2020}, this is given by $(M, s)$.

    By our compactness for type omission, there is a model $N$ of $T^0 \cup T^1$ that omits $p_F(x)$ for each $F:\cP_\kappa\lambda \to \lambda$.  By standard arguments,
    $$U = \left\{X \subset \cP_\kappa\lambda : N \vDash\text{`}d E c_X\text{'} \right\}$$
    is an ultrafilter on $\cP_\kappa \lambda$ so each $[s] \in U$, and the omission of the $p_F$ types implies that $U$ is normal.  Taken together, folklore results imply that $U$ is $\kappa$-complete.  Thus, $\kappa$ is $\lambda$-supercompact.
    \end{proof}

    An important point is that the ultrafilter $U$ contains not just $[\{\alpha\}]$ for each $\alpha \in \cP_\kappa \lambda$ but also $[s]$ for each $s \in \cP_\kappa \lambda$, which is the set the ultrafilter is on.  Then this allows the choice $d = s$ to omit each type.

    This fails when moving to huge cardinals: to imitate this argument, we would want to have $[s] \in U$ for each $s \in [\lambda]^\kappa$.  However, because these $s$ are of size $\kappa$, this is much harder: if $U$ is an ultrafilter derived from a huge embedding, then in fact $[s] \not\in U$.

\footnotei{WB: I thought we had a weak version of this?  Let me find it and see if I can write it up\\
`Notes, \S 2' is where I pointed myself}	
	
	\section{Compactness of $\bL(I)$}\label{sec:compactnessLI}
	
	Magidor and Väänänen showed in \cite{mag2011} that it is consistent that the so called \emph{LST} number of the logic $\bL(I)$ is the first supercompact cardinal. We show that in the model they produce, if the compactness number of $\bL(I)$ exists, it is larger than the first supercompact cardinal.\footnote{After this result was obtained, that it is consistent that the compactness number of $\bL(I)$ is larger than the first supercompact cardinal was also proved by Gitman and the second author in \cite{gitman2024upward} using different methods.}
	
	We will actually work with the logic $\bL(\WF,I)$, which is justified by the following proposition.
	\begin{proposition}
		$\comp(\bL(\WF, I)) = \comp(\bL(I))$.
		\begin{proof}
			Clearly $\comp(\bL(I)) \leq \comp(\bL(\WF,I))$. The so-called \emph{$\Delta$-closure} $\Delta(\bL(I))$ of the logic $\bL(I)$ can define well-foundedness (cf., e.g., \cite[§4.1]{bagaria2016symbiosis}), and so $\comp(\bL(\WF,I)) \leq \comp(\Delta(\bL(I)))$. But it is known that the $\Delta$-closure preserves compactness numbers (cf., e.g., \cite[p. 72]{bar1985}, or -- for a proof -- \cite[Lemma 4.1.9]{osinski2021symbiosis}), and so $\comp(\bL(I)) = \comp(\bL(\WF, I)) = \comp(\Delta(\bL(I)))$. 
		\end{proof}
	\end{proposition}
	We will use some results by Magdior and Väänänen connecting model theory of $\bL(\I)$ to combinatorial principles. More concretely, they reprove results on \emph{good scales} from Shelah's pcf theory (cf. \cite{shelah1994}) for weaker notions tailored to their applications. For this, if $f,g \in \Ord^\omega$, we write
	\[
	f <^* g \text{ iff } f(n) < g(n) \text{ for all but finitely many } n \in \omega.
	\]
	\begin{definition}[{\cite[Definition 12]{mag2011}}]\label{def:goodSeq}
		Let $(f_\alpha \colon \alpha < \mu)$ be a $<^*$-increasing sequence of $f_\alpha \in \Ord^{\omega}$. We make the following conventions.
		\begin{enumerate}
			\item[(a)] An ordinal $\delta \in \mu$ is called a \emph{good point of the sequence $(f_\alpha \colon \alpha < \mu)$} iff there is a cofinal set $C \subseteq \delta$ and a function $C \rightarrow \omega$, $\alpha \mapsto n_\alpha$, such that $\alpha < \beta$ in $C$ and $k > \max(n_\alpha, n_\beta)$ implies $f_\alpha(k) < f_\beta(k)$.
			\item[(b)] The sequenece $(f_\alpha \colon \alpha < \mu)$ is called a \emph{good} if there is a club subset $D \subseteq \mu$ such that all members of $D$ are good points of the sequence. 
		\end{enumerate}
	\end{definition}
	\begin{fact}[Magidor \& Väänänen ({\cite[Theorem 15]{mag2011}})]\label{fact:goodseqLSTLI}
		If $\kappa = \LST(\bL(\I))$, then there is no good sequence $(f_\alpha \colon \alpha < \lambda^+)$ of functions $f: \omega \rightarrow \lambda$ for any $\lambda \geq \kappa$ with $\cof(\lambda) = \omega$. 
	\end{fact}
	Recall that the \emph{Singular Cardinal Hypothesis} (SCH) is the statement:
	\begin{quote}
		If $\lambda$ is singular and $2^{\cof(\lambda)} < \lambda$, then $2^\lambda = \lambda^+$.
	\end{quote}
	Failure of SCH is connected to the existence of good sequences.
	\begin{fact}[Shelah \cite{shelah1994}]\label{fact:ShelahSCHgoodseq}
		If $\lambda$ is a singular cardinal of $\cof(\lambda) = \omega$ such that SCH fails at $\lambda$, then there is a good sequence $(f_\alpha \colon \alpha < \lambda^+)$ of functions $f_\alpha: \omega \rightarrow \lambda$.
        \begin{proof}
             Cf. {\cite[Lemma 17]{mag2011}} for a concise proof in the setting presented here.
        \end{proof}
	\end{fact}
	We show that the existence of good sequences can be expressed by a sentence of $\bL(\WF, \I)$.
	\begin{lemma}\label{lem:goodseqLI}
		There is a vocabulary $\tau$ and a sentence $\varphi_{\text{good}} \in \bL(\WF, \I)[\tau]$ such that for any set $M$, the following are equivalent:
		\begin{enumerate}
			\item[(1)] $|M|$ is a successor cardinal $\lambda^+$ such that $\cof(\lambda) = \omega$ and there is a good sequence $(f_\alpha \colon \alpha < \lambda^+)$ of functions $f_\alpha: \omega \rightarrow \lambda$.
			\item[(2)] There is a $\tau$-structure $\mathcal M$ with universe $M$ such that $\mathcal M \models \varphi_{\text{good}}$. 
		\end{enumerate}
		\begin{proof}
			Consider the vocabulary $\{E, F, G, D, C, S \}$, where $D$ is unary, $E, G, C$ are binary, and $F$ and $S$ are ternary. Then let $\varphi_{\text{good}}$ be the conjunction of the following sentences, where (i) uses $\WF$ for the well-foundedness assertion:
			\begin{enumerate}
				\item[(i)] ``$E$ is a well-order without a largest element."
				\item[(ii)] $\ot(E)$ is a cardinal: $\neg \exists x \I yz (y E x, z=z)$.
				\item[(iii)] There is a largest cardinal $\lambda$: \[\exists \lambda\forall x((xE\lambda \rightarrow \neg \exists w Iyz(yEw, zE\lambda)) \wedge (\lambda E x \rightarrow \I yz(yE\lambda,zEx))).\]
				\item[(iv)] ``$G(\cdot, \cdot)$ is a function with domain the smallest limit ordinal (i.e., $\omega$) into $\lambda$ with unbounded range."
				\item[(v)] ``For every $\alpha$, $F(\alpha, \cdot,\cdot)$ is a function with domain $\omega$ and range $\subseteq \lambda$."
				\item[(vi)] ``If $\alpha < \beta$, there is an $n < \omega$ such that $\forall m \geq n$: $F(\alpha,m) < F(\beta, m)$."
				\item[(vii)] ``$D$ is a club subset of the model and $\forall \delta \in D$, $C(\delta,\cdot)$ is good, i.e., $C(\delta, \cdot)$ is a club subset of $\delta$ and $S(\delta, \cdot, \cdot)$ is a function with domain $C(\delta, \cdot)$  and range $\omega$ such that $\forall \alpha < \beta$ both in $C(\delta,\cdot)$: $\forall k > \max(S(\delta, \alpha), S(\delta, \beta))(F(\alpha, k) < F(\beta, k))$."
			\end{enumerate}
			Then $\varphi_{\text{good}}$ is as desired. For (vii), note that we can express that, for example, $D$ is club by saying: 
			\[
			\begin{tabular}{ l l }
				$\forall x \exists y (D(y) \wedge x E y)$ & (``$D$ is unbounded") \\
				$\forall x (\forall y(yEx \rightarrow \exists z(yEz \wedge zEx \wedge D(z)) \rightarrow D(x))$ & (``$D$ is closed")
			\end{tabular}
			\]
		\end{proof}
	\end{lemma}
	\begin{lemma}\label{lem:compLIgoodseq}
		If for every $\kappa < \delta = \comp(\bL(\WF,\I))$ there is $\lambda \geq \kappa$ of cofinality $\omega$ such that there is a good sequence $(f_\alpha \colon \alpha < \lambda^+)$ of functions $f_\alpha: \omega \rightarrow \lambda$, then there are unboundedly many such $\lambda$ in the ordinals.
		\begin{proof}
			Let $\rho$ be any cardinal. It suffices to show that there is $\gamma \geq \rho$ of cofinality $\omega$ such that there is a good sequence of functions $(f_\alpha \in \gamma^\omega \colon \alpha < \gamma^+)$. Consider the theory $T = \{\varphi_{\text{good}}\} \cup \{c_i \neq c_j \colon i < j < \rho^+ \}$, where the $c_i$ are new constants. Let $T_0 \subseteq T$ be of size $|T_0| < \delta$. Then for some $\kappa < \delta$, there are $\kappa$-many sentences of the form $``c_i \neq c_j"$ appearing in $T_0$. By the assumption take $\lambda \geq \kappa$ such that there is a good sequence $(f_\alpha \in \lambda^\omega \colon \alpha < \lambda^+)$. Then by Lemma \ref{lem:goodseqLI}, $\varphi_{\text{good}}$ has a model with universe $\lambda^+$, so of size $> \kappa$. Clearly, this can be expanded to a model of $T_0$. We showed that $T$ is $< \delta$-satisfiable. Hence, $T$ has a model, which gives rise to a model $M$ of $\varphi_{\text{good}}$ of size $\geq \rho^+$. Then with $|M| = \gamma^+$, $\gamma$ is as desired, by Lemma \ref{lem:goodseqLI}.
		\end{proof}
	\end{lemma}
	\begin{corollary}\label{cor:clIabovesc}
		If there is an extendible cardinal, then in a forcing extension of the universe it holds that $\comp(\bL(\I))$ is larger than the first supercompact cardinal.
		\begin{proof}
			Let $\eta$ be the smallest extendible cardinal and $\nu$ the smallest supercompact one. We can use Magidor's forcing from \cite[Section 4]{mag1976} to go to a model $N$ in which $\nu$ becomes simultaneously the smallest supercompact and the smallest strongly compact cardinal, by introducing unboundedly many points $(\lambda_i)_{i < \nu}$ below $\nu$ of cofinality $\cof(\lambda_i) = \omega$ such that SCH fails at $\lambda_i$. Because the forcing used has size $< \eta$, $\eta$ remains extendible in $N$. Let us work in $N$. It is known that the smallest supercompact $\nu = \LST(\bL^2)$ (cf. \cite{mag1971}). In particular, $\LST(\bL(\I)) \leq \LST(\bL^2) \leq \nu$. Thus, by Fact \ref{fact:goodseqLSTLI}, for no $\lambda \geq \nu$ of cofinality $\omega$ there is a good sequence $(f_\alpha \colon \alpha < \lambda^+)$ of functions $f_\alpha: \omega \rightarrow \lambda$. Furthermore, by Fact \ref{fact:ShelahSCHgoodseq}, for all $\lambda_i$, $i < \nu$, there is such a good sequence. Thus, using Lemma \ref{lem:compLIgoodseq}, it is impossible that $\comp(\bL(\WF, \I)) \leq \nu$. Because the smallest extendible is equal to $\comp(\bL^2)$ (cf. \cite{mag1971}), also $\comp(\bL(I)) = \comp(\bL(\WF, \I)) \leq \comp(\bL^2)$ exists and is thus larger than $\nu$. 
		\end{proof} 
	\end{corollary}
	
	\section{Upward LST numbers of second-order and sort logic}\label{sec:ULST}
	
	Motivated by the LST number, Galeotti, Khomskii, and Väänänen considered upward Löwenheim-Skolem properties of logics \cite{gal2020}. 
	\begin{definition}[Galeotti, Khomskii \& Väänänen \cite{gal2020}]
		The \emph{upward Löwenheim-Skolem-Tarski} (ULST) number $\ULST(\cL)$ of a logic $\cL$, should it exist, is the smallest cardinal $\kappa$ such that for any language $\tau$ and $\varphi \in \cL[\tau]$, if there is a $\tau$-structure $\mathcal A$ satisfying $\varphi$ and such that $|A| \geq \kappa$, then for any cardinal $\lambda$ there is a superstructure $\mathcal B \supseteq \mathcal A$ of size $|B| \geq \lambda$ satisfying $\varphi$. 
	\end{definition}
	This is a strengthening of the classically studied \emph{Hanf number}, which omits the superstructure condition. The Hanf number does not carry any large cardinal strength (cf. \cite[Chapter II, Theorem 6.1.4]{bar1985}). On the contrary, Galeotti, Khomskii, and Väänänen showed that the $\ULST(\bL^2)$ is $n$-extendible for any natural number $n$ (cf. \cite[Corollary 7.7]{gal2020}). They conjectured that $\ULST(\bL^2)$ is exactly the first extendible cardinal (cf. \cite[Conjecture 7.8]{gal2020}). We will confirm their conjecture (Theorem \ref{thm:ULSTSOL}). Moreover, we will show that the $\ULST(\bL^{s,n})$ is exactly the first $C^{(n)}$-extendible cardinal.\footnote{These results were also obtained by Gitman and the second author, using an entirely different argument (cf. \cite[Theorem 6.1]{gitman2024upward}). Yair Hayut found yet another proof of Theorem \ref{thm:ULSTSOL}, which remains unpublished.} 
	
	We will make use of \emph{truth predicates}. We assume that formulas in the language of set theory are coded as tuples of natural numbers. If $(M,\in)$ is a transitive set closed under the pairing function, a subset $T \subseteq M$ is called a \emph{truth predicate for $(M, \in)$} if for any first-order formula $\varphi(x_1, \dots, x_n)$ in the language of set theory, and any $a_1, \dots, a_n \in M$:
	\[
	(M, \in) \models \varphi(a_1, \dots, a_n) \text{ iff } (\varphi, a_1, \dots, a_n) \in T. 
	\]
	Magidor noted (cf. \cite{mag1971}) that there is a first-order sentence $\varphi_{\text{truth}}$ in the language $\{\in, T\}$ with an additional unary predicate $T$, such that for any transitive $(M,\in)$ and any $T^M \subseteq M$, 
	\[
	(M,\in,T^M) \models \varphi_{\text{Truth}} \text{ iff } T^M \text{is a truth predicate for } (M, \in).
	\]
	Further, we can consider the language $\{\in, T, c, S, P \}$ with an additional constant $c$, and predicates $S$ and $P$. We may fix sentences $\varphi_\emptyset$, $\varphi_{\text{Succ}}$, and $\varphi_{\text{Pair}}$ coding that in a transitive structure $\mathcal M = (M, \in, T^M, c^M, S^M, P^M)$, $c^M$ is the emptyset, $S^M$ codes the successor function, and $P^M$ codes the pairing function, so, for example, 
	\[
	\mathcal M \models \varphi_{\text{Pair}} \text{ iff } \forall x,y,z \in M:(P^M(x,y,z) \leftrightarrow z = (x,y)). 
	\] 
	It is easy to see that if $j: (M,\in, T^M,c^M, S^M,P^M) \rightarrow (N,\in, T^N, c^N, S^N, P^N)$ is an (not necessarily elementary) embedding between transitive structures satisfying the sentences $\varphi_{\text{Truth}}, \varphi_\emptyset, \varphi_{\text{Succ}}$, and $\varphi_{\text{Pair}}$, then $j$ is an elementary embedding between the structures $(M, \in)$ and $(N, \in)$. For details, consider \cite[Section 4.3]{osinski2024}. We are now ready to prove our theorem.
	\begin{theorem}\label{thm:ULSTSOL}
		Let $\delta$ be a cardinal. Then $\delta$ is the smallest extendible cardinal if and only if $\delta = \ULST(\bL^2)$. 
	\end{theorem}
	\begin{proof}
		First note that if $\delta$ is extendible, then $\delta = \comp(\bL^2)$ (cf. \cite{mag1971}), and using this, it is easy to see that $\ULST(\bL^2) \leq \comp(\bL^2) \leq \delta$.\footnote{That $\ULST(\cL) \leq \comp(\cL)$ for any logic $\cL$ can be derived analogously to the standard proof of the upward Löwenheim-Skolem theorem from the Compactness theorem for first-order logic (cf., e.g., \cite[Proposition 3.1]{gitman2024upward} where this is carried out).} It is therefore sufficient to show that if $\delta = \ULST(\bL^2)$, then there is an extendible cardinal $\leq \delta$. 
		
		Consider any ordinal $\alpha\geq\delta$ of cofinality $\omega$. Fix a function $f_\alpha$ with domain $\omega$ that is cofinal in $\alpha$, a truth predicate $T_\alpha$ for $(V_\alpha, \in)$ and relations $S_\alpha$ and $P_\alpha$ coding the successor and pairing functions. Then the structure $\mathcal A = (V_\alpha, \in, f_\alpha, T_\alpha, \emptyset, S_\alpha, P_\alpha)$ is a model of the conjunction of the following sentences in the language $\{ \in, f, T, c, S, P\}$, where $f$ is a two place predicate.
		\begin{enumerate}
			\item[(i)] Magidor's $\Phi$.
			\item[(ii)] $\varphi_{\text{Truth}} \wedge \varphi_\emptyset \wedge \varphi_{\text{Succ}} \wedge \varphi_{\text{Pair}}$.
			\item[(iii)] ``$f$ is a function with domain $\omega$ which is cofinal in the ordinals":
			\[
			\begin{split}
				&\forall x,y,z (f(x,y) \wedge f(x,z)\rightarrow y=z \wedge \Ord(y)) \wedge \\ & \forall x(x \in \omega \leftrightarrow \exists yf(x,y)) \wedge
				\forall \alpha (\Ord(\alpha) \rightarrow \exists x,\beta(\alpha < f(x,\beta))).
			\end{split}
			\]
		\end{enumerate}
		Because $\delta = \ULST(\cL^2)$ we find a superstructure $\mathcal A_\alpha^* = (A_\alpha, E_\alpha, f_\alpha^*, T_\alpha^*, c_\alpha^*, S_\alpha^*, P_\alpha^*)$ of size $> |V_\alpha|$ satisfying the above sentences (i) to (iii). By $\Phi$, we can collapse $\mathcal A_\alpha$ to a structure of the form $\mathcal A_\alpha^* = (V_{\beta_\alpha}, \in, f_{\beta_\alpha}, T_{\beta_\alpha}, c_{\beta_\alpha}, S_{\beta_\alpha, P_{\beta_\alpha}})$. Then the collapse isomorphism restricts to an embedding $j: \mathcal A \rightarrow \mathcal A_\alpha^*$. By (ii) and the comments above, $j$ is an elementary embedding $j_\alpha: (V_\alpha, \in) \rightarrow (V_{\beta_\alpha}, \in)$. By (iii), $f_{\beta_\alpha}$ is a function with domain $\omega$ which is cofinal in $\beta_\alpha$. Notice that $j_\alpha(f_\alpha(n)) = f_{\beta_\alpha}(j_\alpha(n)) = f_{\beta_\alpha}(n)$. Thus for some $n$, $j_\alpha(f_\alpha(n)) > f_\alpha(n)$ and so $j_\alpha$ has some critical point $\crit(j_\alpha) < \alpha$. Therefore, the function $F$ which sends $\alpha$ to the smallest value of a critical point $\crit(j_\alpha)$ of some elementary embedding $j_\alpha: V_\alpha \rightarrow V_{\gamma}$ is a definable function on the stationary class $S = \{\alpha > \delta \colon \cof(\alpha) = \omega\}$. We may use a version of Fodor's Lemma which shows that any regressive function on a definable stationary proper class of ordinals is constant on a unbounded class (cf. \cite{git2019}). Hence, $F$ is constant on an unbounded subclass of $S$, say with value $\kappa$. Then $\kappa$ is the critical point of $j_\alpha: (V_\alpha, \in) \rightarrow (V_\gamma, \in)$ for a proper class of $\alpha$ and therefore extendible.
		
		Hence, we can let $\eta$ be the smallest extendible cardinal.  We claim $\eta \leq \delta$. Suppose $\eta > \delta$. Notice that the assertion that $\delta$ is a $\ULST$ number of $\cL^2$ can be formalized by the following $\Pi_3$ formula:
		\[
		\forall \mathcal A \forall \varphi \forall \bar \gamma \exists \mathcal B(|\mathcal A| \geq \delta \wedge \mathcal A \models_{\cL^2} \varphi \rightarrow \mathcal A \subseteq \mathcal B \wedge |\mathcal B|\geq \bar \gamma \wedge \mathcal B \models_{\cL^2} \varphi ).
		\]
		This is a $\Pi_3$-statement as $\models_{\cL^2}$ is $\Delta_2$-definable (cf., e.g., \cite[Proposition 3.6]{gal2020}). By extendibility, $\eta \in C^{(3)}$ (cf. \cite[Proposition 23.10]{kan}) and so $\eta$ satisfies that $\delta = \ULST(\cL^2)$. Then we can repeat our argument in $V_\eta$ and find a cardinal $\nu < \eta$ such that $V_\eta \models ``\nu$ is extendible". Because also being extendible is a $\Pi_3$-statement, $V_\eta$ is correct about this fact, and so $\nu$ is really extendible. But this contradicts minimality of $\eta$. 
	\end{proof}
	Using that $\bL^{s,n}$ is able to express that some ordinal is in $C^{(n)}$, the above proof can be adapted to show the following theorem.
	\begin{theorem}\footnote{Again, this was also shown by different methods in \cite{gitman2024upward}.}
		A cardinal $\kappa$ is the first $C^{(n)}$-extendible iff $\kappa = \ULST(\bL^{s,n})$. 
	\end{theorem}
	\begin{corollary}
		VP is equivalent to the schema ``Every logic has a ULST number."
		\begin{proof}
			``Every logic has a ULST number." implies that for any $n$, $\bL^{s, n}$ has a ULST number. By the previous theorem, this shows that for any $n$ there is a $C^{(n)}$-extendible cardinal, and so VP holds by Bagaria's result \ref{fact:BagariaMakowsky}. 
			
			On the other hand, by Makowsky's \ref{fact:BagariaMakowsky}, VP implies that every logic has a compactness number which in turn implies that every logic has a ULST number.
		\end{proof}
	\end{corollary}
    We summarize the known connections between VP and model theory in Figure 2 below. Here the term \emph{$n$-SHC numbers} refers to the compactness properties for strong Henkin models considered in \cite{osinski2024Henkin}. This naming follows \cite{osinski2024}. The characterizations of extendible and {$C^{(n)}$-extendible} cardinals by ULST properties provided here are highlighted in red.

    \begin{figure}[h]\label{fig2}
	\centering
	\noindent \includegraphics[width=0.99\hsize]{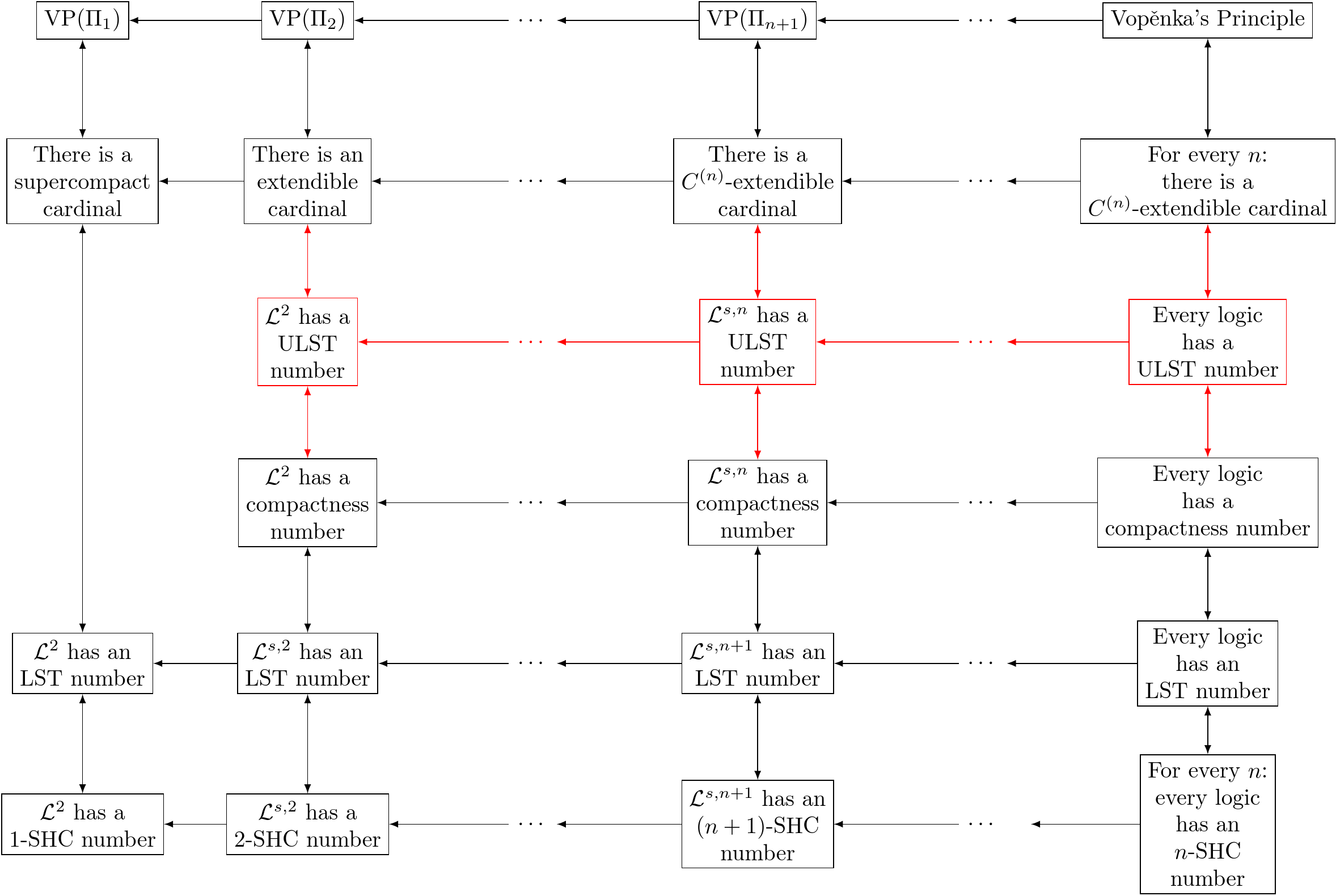}
	\caption{Relations between $\VP$, $C^{(n)}$-extendible cardinals, ULST numbers, compactness numbers, LST numbers, and SHC numbers.}
    \end{figure}

	\bibliography{bibliography}{}
	\bibliographystyle{alpha}
\end{document}